\documentclass{amsproc}
\usepackage{amsmath,amssymb,latexsym,amsfonts,cite,hyperref, enumerate, fullpage}
\usepackage[all]{xy}

\hypersetup{unicode=false, pdftoolbar=true, pdfmenubar=true, pdffitwindow=false, pdfstartview={FitH}, pdfnewwindow=true, colorlinks=false, linkcolor=blue, citecolor=green}

\newtheorem{theorem}{Theorem}[section]

\newtheorem{corollary}[theorem]{Corollary}
\newtheorem{lemma}[theorem]{Lemma}

\newtheorem{definition}[theorem]{Definition}
\numberwithin{equation}{section}
\theoremstyle{definition}

\begin{document}
\title[Amalgamated Products]{Amalgamated Products of Ore and Quadratic Extensions of Rings}
\author{Garrett Johnson}
\address{Department of Mathematics, North Carolina State University, Raleigh, NC 26795-8205}
\email{gwjohns3@ncsu.edu}

\begin{abstract}
We study the ideal theory of amalgamated products of Ore and quadratic extensions over a base ring $R$. We prove an analogue of the Hilbert Basis theorem for an amalgamated product $Q$ of quadratic extensions and determine conditions for when the one-sided ideals of $Q$ are principal or doubly-generated. We also determine conditions that make $Q$ a principal ideal ring.  Finally, we show that the double affine Hecke algebra $\mathbb{H}_{q,t}$ associated to the general linear group $GL_2(k)$ (here, $k$ is a field with $\text{char}(k)\neq 2$) is an amalgamated product of quadratic extensions over the three-dimensional quantum torus ${\mathcal O}_{\bf q}((k^\times)^3)$ and give an explicit isomorphism. In this case, it follows that $\mathbb{H}_{q,t}$ is a noetherian ring.
\end{abstract}

\maketitle

\section{Introduction}

Since the appearance of the seminal paper of Ore \cite{Ore} in 1933, Ore extensions $R[x;\tau,\delta]$ (or \emph{skew-polynomial rings}) have played an important roll in several areas of algebra, such as the universal enveloping algebras of solvable Lie algebras, group rings of polycyclic-by-finite groups, quantized coordinate rings, and rings of differential operators. The Ore extensions of a given ring $R$ behave like an ordinary polynomial ring $R[x]$, except the coefficients do not necessarily commute past the variable $x$. The noncommutativity  is governed by a ring endomorphism $\tau:R\to R$ and a (left) $\tau$-derivation $\delta:R\to R$. In an Ore extension based on these data $xr=\tau (r)x+\delta(r)$  ($r\in R)$. 

Let $S_1=R[x;\tau_1,\delta_1]$ and $S_2=R[y;\tau_2,\delta_2]$ be Ore extensions over a base ring $R$, and let $S:=S_1*_RS_2$, the pushout (or \emph{amalgamated free product}) ring. Primeness of the base ring $R$ carries over to the amalgamated product $S$. In Section \ref{Ore Extensions} we prove

\begin{theorem} Let $\tau_1$ and $\tau_2$ be autmorphisms of a ring $R$. If $R$ is a $(\tau_1,\tau_2)$-prime ring (i.e. if the product of two $(\tau_1,\tau_2)$-stable ideals is zero, then at least one ideal is zero), then $S$ is a prime ring.\end{theorem}

Since the notion of $(\tau_1,\tau_2)$-prime is weaker than the usual notion of prime, it immediately follows that if $R$ is prime, then $S$ is also prime (Theorem \ref{Rprime->Sprime}). Secondly, we show a prime $(\tau_1,\tau_2,\delta_1,\delta_2)$-ideal $I$ in a noetherian ring $R$ will generate a prime ideal in $S$, and the factor ring $S/I$ is isomorphic to the amalgamated product $(R/I)\left[ x,\tau_1,\delta_1\right]*_{R/I}(R/I)\left[y,\tau_2,\delta_2\right]$. Certain localizations of $S$ reduce to studying the case where the base ring $R$ is a division ring.

\begin{theorem} Let $\tau_1$ and $\tau_2$ be automorphisms of a noetherian domain $R$, and let $ X:=R\backslash\{0\}$. Then
\begin{enumerate}[(i)]
\item $X$ is a right denominator set in $S$, and
\item the ring of fractions $SX^{-1}$ is isomorphic to the amalgamated product 
\[(R X^{-1})[x;\tau_1,\delta_1]*_{R X^{-1}}(R X^{-1})[y;\tau_2,\delta_2],\]
\end{enumerate}
where $\tau_1,\tau_2,\delta_1,\delta_2$ denote the induced maps on the division ring of fractions $R X^{-1}$.
\end{theorem}

In Section \ref{Quadratic Extensions}, we study the amalgamated product $Q_1*_RQ_2$ of quadratic extensions of a base ring $R$. We call an extension of rings $R\subseteq A$ \emph{quadratic} if there exists $x\in A\backslash R$ so that $A$ is a free left $R$-module with basis $\{1,x\}$. A quadratic extension of $R$ will necessarily be isomorphic to a factor ring of an Ore extension $R[x;\tau,\delta]$. More precisely, $A\cong R[x;\tau,\delta]/I$, where $I$ is a principal ideal in $R[x;\tau,\delta]$ generated by an element of the form $x^2+ax+b$ ($a,b\in R$). Let $Q_1=S_1/I_1$ and $Q_2=S_2/I_2$ be quadratic extensions over $R$, and put $Q=Q_1*_RQ_2$. Under certain conditions, the ideal theory of $Q$ turns out to have a nice description.  In particular, we have an analogue of the Hilbert Basis theorem.

\begin{theorem}
If $\tau_1$ and $\tau_2$ are automorphisms of a noetherian base ring $R$, then $Q$ is noetherian.
\end{theorem}

Furthermore, we prove the following.
\begin{theorem}
\label{almostPID}If the base ring $R$ is a division ring and $\tau_1$ and $\tau_2$ are automorphisms, then all one-sided ideals of $Q$ are either principal or doubly generated.
\end{theorem}

Theorem \ref{almostPID} above is reminiscent of a famous result of Stafford. In \cite{Stafford}, Stafford proves the Weyl algebras $A_n(k):=k\left<x_1,...,x_n,\frac{\partial}{\partial x_1},...,\frac{\partial}{\partial x_n}\right>$ over a field $k$ of characteristic zero exhibit this property: every one-sided ideal of $A_n(k)$ is principal or doubly generated. Placing extra restrictions on the base ring $R$ and the skew derivations $\delta_1$, $\delta_2$ forces the amalgamated product $Q$ to be a principal ideal ring.

\begin{theorem}
If $\tau_1$  and $\tau_2$ are automorphisms of a division ring $R$ and the skew derivations $\delta_1$, $\delta_2$ are not $\tau_1$ (resp $\tau_2$)-inner, then $Q$ is a principal ideal ring.
\end{theorem}

One application to studying these types of extensions and their amalgamated products comes from the theory of double affine Hecke algebras. The double affine Hecke algebras are algebras related to symmetric polynomials and were introduced by Cherednik in the early 1990's \cite{Ch}. They were instrumental to the proof of the Macdonald constant-term conjectures \cite{Mac}. 

For a reductive algebraic group $G$ with rank $n$ over a field $k$, the DAHA $\mathbb{H}_{q,t}(G)$ associated to $G$ is a deformation (with deformation parameters $q_1,...,q_n,t_1,...,t_n$) of the group algebra of the \emph{extended double affine Weyl group} $W\ltimes (P\oplus P)$, where $W$ is the Weyl group of $G$ and $P$ is the weight lattice. When the deformation parameters are all specialized to $1\in k$, we recover the group algebra $k[W\ltimes (P\oplus P)]$. 

In \cite{Gehles}, Gehles studies the structure theory of the DAHAs and their associated trigonometric and rational degenerations and proves the DAHAs are noetherian when the deformation parameters $q_1,...,q_n$ are specialized to $1\in k$ (but the $t_i$'s may be arbitrarily chosen) \cite[Corollary 2.1.9]{Gehles}. It remains an open problem to prove noetherianity for arbitrary deformation parameters. 

Let $k$ be a field with $\text{char}(k)\neq 2$ and let ${\mathcal O}_{\bf q}((k^\times)^3)$ denote the $k$-algebra generated by invertible variables $z_1,z_2,z_3$ and subject to the defining relations $z_1z_2=z_2z_1$, $z_1z_3=q^{-1}z_3z_1$, and $z_2z_3=q^{-1}z_3z_2$. The ring ${\mathcal O}_{\bf q}((k^\times)^3)$ is commonly referred to as (the quantized coordinate ring of) a three-dimension quantum torus. In Section \ref{DAHA}, we prove $\mathbb{H}_{q,t}(GL_2(k))$ is isomorphic to an amalgamated product of quadratic extensions over ${\mathcal O}_{\bf q}((k^\times)^3)$.   In Theorem \ref{DAHA as an amalgamated product}, we give an explicit isomorphism. Here, the deformation parameters $q$ and $t$ may be chosen to be arbitrary nonzero scalars in $k$ such that $t^{1/2}$ exists. Since ${\mathcal O}_{\bf q}((k^\times )^3)$ is known to be noetherian, it follows that  $\mathbb{H}_{q,t}(GL_2(k))$ is also noetherian. 
\\

{\bf Acknowledgments.} This paper is based on some work from my Ph.D. thesis \cite{J} written at the University of California at Santa Barbara. I am grateful to Ken Goodearl and Milen Yakimov for suggesting the problems that motivated this work.

\section{Amalgamated Products of Rings}

\setcounter{theorem}{0}
In this section, we prove some general results concerning the structure theory of amalgamated products of quadratic extensions and Ore extensions of rings. The results are then applied to the example of the double affine Hecke algebra $\mathbb{H}_{q,t}(GL_2(k))$. 

\begin{definition}\label{def: amalgamated product}
Let $\beta_1:R\to S_1$ and $\beta_2:R\to S_2$ be ring homomorphisms. 
The \emph{amalgamated product of $S_1$ and $S_2$ along $R$} is a triple $(S,\phi_1,\phi_2)$ where  $\phi_1:S_1\to S$ and $\phi_2:S_2\to S$ are ring homomorphisms satisfying $\phi_1\beta_1 = \phi_2\beta_2$ and the universal property: Given any ring $S^\prime$ with homomorphisms $\psi_1:S_1\to S^\prime$, $\psi_2:S_2\to S^\prime$ satisfying $\psi_1\beta_1=\psi_2\beta_2$, there exists a unique homomorphism $\theta :S\to S^\prime$ so that the following diagram commutes.
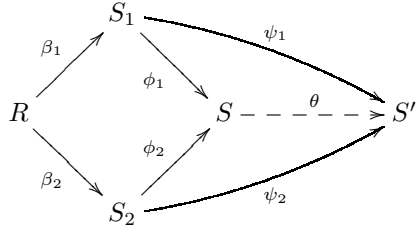
\begin{center}\xymatrix{
&&&&&&S_1\ar[dr]_{\phi_1}\ar@/^/[drrr]^{\psi_1}&&\\
&&&&&R\ar[ur]^{\beta_1}\ar[dr]_{\beta_2}&&S\ar@{-->}[rr]^{\theta}&&S^\prime\\
&&&&&&S_2\ar[ur]^{\phi_2}\ar@/_/[urrr]_{\psi_2}&&&}\end{center}
\end{definition}

In all examples we consider, $R$ is a subring of $S_i$ ($i=1,2$) and $\beta_i:R\to S_i$ are inclusion maps. When referring to an amalgamated product, we will usually not mention the ring $R$ and simply say $S$ is an amalgamated product of $S_1$ and $S_2$. It is well-known that an amalgamated product $S$ of $S_1$ and $S_2$ exists and is unique up to isomorphism. Thus, we will call $S$ \emph{the} amalgamated product and denote it by $S=S_1*_RS_2$. The presentation we will use for $S_1*_RS_2$ is given by the generating set $R\sqcup S_1\sqcup S_2$.  There are two main types of relations among the generators. First of all, if $r$ and $s$ are both in $R$, $S_1$, or $S_2$, then their product in $S_1*_RS_2$ is the same as their product in the appropriate ring. The other defining relations are the ``cross-relations": for every $r\in R$, $s_1\in S_1$, and $s_2\in S_2$,
\begin{enumerate}[(i)]
\item $rs_1:=$ the product $\beta_1(r)\cdot s_1$ in $S_1$, 
\item $s_1r:=$ the product $s_1\cdot\beta_1(r)$ in $S_1$,
\item $rs_2:=$ the product $\beta_2(r)\cdot s_2$ in $S_2$,
\item $s_2r:=$ the product $s_2\cdot\beta_2(r)$ in $S_2$,
\item $r=\beta_1 (r) = \beta_2(r)$.
\end{enumerate}

\subsection{Amalgamated Products of Ore extensions}\label{Ore Extensions}

One example we consider is the amalgamated product of Ore extensions over a base ring $R$. Recall, an \emph{Ore extension over $R$} with left-hand coefficients is a ring $A$ satisfying the following conditions:
\begin{enumerate}[(i)]
\item $A$ contains $R$ as a subring,
\item there exists $x\in A$ so that $A$ is a free left $R$-module having basis $\{1,x,x^2,x^3,...\},$ and 
\item $xR\subseteq Rx+R$.
\end{enumerate}

From the definition, it follows that there exists a ring endomorphism $\tau:R\to R$ and a left $\tau$-derivation $\delta:R\to R$ (i.e. a $\mathbb{Z}$-linear map satisfying $\delta (rs) = \tau(r)\delta(s)+\delta(r)s$ for all $r,s\in R$) so that 
\begin{equation}\label{xr=...}xr=\tau (r)x+\delta(r)\end{equation}
for every $r\in R$. Given any such pair of maps $(\tau, \delta)$, a corresponding Ore extension exists and is unique up to isomorphism. We denote it by $R[x;\tau,\delta]$. 
\\

Throughout the remainder of this section, $\tau_1$ and $\tau_2$ will denote automorphisms of a ring $R$, and $\delta_1$ and $\delta_2$ are left $\tau_1$ (resp. $\tau_2$)-derivations of $R$.
\\

We let $S_1=R[x;\tau_1,\delta_1]$, $S_2=R[y;\tau_2,\delta_2]$, and put $S=S_1*_RS_2$. It follows from Eqn. \ref{xr=...} that the set of words in the letters $x$ and $y$ are a spanning set (over $R$) for $S$. (In fact, the set of words is an $R$-basis of $S$. One can verify $R$-linear independence by using the standard argument: let $E$ be the semigroup ring over $R$ of the free semigroup on two letters $X$ and $Y$, viewed as a left $R$-module, and construct an $S$-module structure on $E$ that mimics left multiplication in $S$. Finally consider the image of $1\in E$ by the action of an arbitrary element $s\in S$.)  There are two natural ways to order words, alphabetically and by length. Ordering first by length, then alphabetically (if lengths equal) gives us a total ordering on words. Thus, every nonzero $s\in S$ may be written in the form $s=rw+[\text{lower terms]}$ for some nonzero $r\in R$ and word $w$. We call $rw$ the \emph{leading term} of $s$, and $r$ is the \emph{leading coefficient}. 

We need to recall the notion of $\eta$-stable ideals and establish some notation used in the results that follow. Let $R^R$ be the set of all functions from a ring $R$ to itself, and let $\mathbb{X}$ be a subset of $R$. If $\eta\subseteq R^R$ and $f(\mathbb{X})\subseteq\mathbb{X}$ for all $f\in\eta$, then $\mathbb{X}$ is called \emph{$\eta$-stable}. An ideal $I\subseteq R$ is called an \emph{$\eta$-ideal} if $I$ is $\eta$-stable. A proper ideal $P\subseteq R$ is \emph{$\eta$-prime} if for any pair of $\eta$-ideals $I,J$ of $R$ with $IJ\subseteq P$, we have either $I\subseteq P$ or $J\subseteq P$. A ring $R$ is called \emph{$\eta$-prime} if $0$ is an $\eta$-prime ideal of $R$. In what follows, we let $\mathbb{Z}_{\geq 0}^{fin. seq.}$ ($\mathbb{Z}^{fin. seq.}$) denote the set of sequences in $\mathbb{Z}_{\geq 0}$ (and $\mathbb{Z}$ resp.) having finitely many nonzero terms. For such a sequence ${\bf j}=(j_1,j_2,j_3...)$, we use the notation $\tau^{\bf j}$ to mean the automorphism $\tau_1^{j_1}\tau_2^{j_2}\tau_1^{j_3}\tau_2^{j_4}\tau_1^{j_5}\cdots$.

\begin{theorem} If $R$ is a domain,  then $S$ is a domain also.\end{theorem}
\begin{proof}
Assume $s_1,s_2\in S$ are nonzero. Therefore, we may write them in the form $s_1 = r_Iw_I+[\text{lower terms}]$, $s_2 = r_Jw_J+[\text{lower terms}]$, for some $r_I,r_J\in R$ nonzero. Thus, 
\begin{equation*}s_1s_2 = r_Ir_J^\prime w_Iw_J+[\text{lower terms}],\end{equation*} 
where $r_J^\prime$ satisfies the identity $w_Ir_J=r_J^\prime w_I+[\text{lower terms}]$. Since $r_J^\prime = \tau^{\bf j}(r_J)$ for some sequence ${\bf j}$ that depends on $w_I$, this implies $r_J^\prime$ is nonzero. Hence the leading term of $s_1s_2$ is nonzero. Thus, $s_1s_2\neq 0$. 
\end{proof}

The following lemma will be useful in the proof of Theorem \ref{Rprime->Sprime}.

\begin{lemma}\label{equivalent conditions 1} 
The following are equivalent:
\begin{enumerate}[(i)]
\item $R$ is $(\tau_1,\tau_2)$-prime
\item For all nonzero $r,r^\prime\in R$, there exist $s\in R$ and a finite sequence ${\bf j}\in\mathbb{Z}^{fin. seq}$ so that $rs\tau^{\bf j}(r^\prime)\neq 0$.
\end{enumerate}
\end{lemma}

\begin{proof}
First, we assume that $R$ is $(\tau_1,\tau_2)$-prime. Let $r,r^\prime\in R$ be nonzero. The smallest $(\tau_1,\tau_2)$-ideals containing $r$ and $r^\prime$ are
\begin{align*}
&\displaystyle{I_{r^{\phantom{\prime}}}=\sum_{{\bf j}\in\mathbb{Z}_{_{\geq 0}}^{^{^{fin. seq.}}}} R(\tau^{\bf j}(r))R} & & \text{ and } & & \displaystyle{I_{r^\prime}=\sum_{{\bf j}\in\mathbb{Z}_{_{\geq 0}}^{^{^{fin. seq.}}}} R(\tau^{\bf j}(r)^\prime)R}
\end{align*}
respectively. Since $I_rI_{r^\prime}\neq 0$, we have $\tau^{\bf j}(r)r^{\prime\prime}(\tau^{\bf j^\prime}(r^\prime))\neq 0$ for some $r^{\prime\prime}\in R$ and  ${\bf j},{\bf j^\prime}\in\mathbb{Z}_{\geq 0}^{fin. seq.}$. Applying the inverse of $\tau^{\bf j}$ to both sides of $\tau^{\bf j}(r)r^{\prime\prime}(\tau^{\bf j^\prime}(r^\prime))\neq 0$ yields condition $(ii)$. Conversely, assume that condition $(ii)$ holds and let $I,J$ be nonzero $(\tau_1,\tau_2)$-ideals of $R$. Pick $r\in I$ and $r^\prime\in J$ both nonzero. There exist $r^{\prime\prime}\in R$ and finite sequences ${\bf j},{\bf j^\prime}\in\mathbb{Z}_{\geq 0}^{fin. seq.}$ so that $\tau^{\bf j}(r)r^{\prime\prime}(\tau^{\bf j^\prime}(r^\prime))\neq 0$. Since $\tau^{\bf j}(r)\in I$ and $r^{\prime\prime}\tau^{\bf j^\prime}(r^\prime)\in J$, the product $IJ$ is nonzero.
\end{proof}

\begin{theorem} If $R$ is $(\tau_1,\tau_2)$-prime, then $S$ is prime.\end{theorem}

\begin{proof} Suppose $f,g\in S$ and there exist nonzero $r,s\in R$ so that $f=rw_I+[\text{lower terms}]$ and $g=sw_J+[\text{lower terms}]$.  By Lemma \ref{equivalent conditions 1}, there exist $s^\prime\in R$ and a sequence ${\bf j}$ so that $rs^\prime\tau^{\bf j}(s)\neq 0$. Furthermore, for every $r^\prime\in R$ and word $w\in S$, we have
\begin{align*}
 fr^\prime wg &= (rw_I)(r^\prime w)(sw_J)+[\text{lower terms}]\\
 &=(\underbrace{rr_I^\prime s_{w,I}}_{\in R}w_Iww_J+[\text{lower terms}],
\end{align*}
where
$r_I^\prime,s_{w,I}\in R$ satisfy $w_Ir^\prime=r_I^\prime w_I+[\text{lower terms}]$ and $w_Iws=s_{w,I}w_Iw+[\text{lower terms}]$.
In particular, by choosing $r^\prime$ so that $r_I^\prime=s^\prime$ and $w$ so that $s_{w,I}=\tau^{\bf j}(s)$, it follows that $fSg\neq 0$. Hence $S$ is prime.
\end{proof}

Since $(\tau_1,\tau_2)$-primeness of a ring is weaker than the usual notion of primeness, we have

\begin{theorem}\label{Rprime->Sprime} If $R$ is prime, then $S$ is prime also.\end{theorem}

The following theorem and the corollary that follows tell us when prime ideals of $R$ generate prime ideals in $S$.

\begin{theorem} If $I$ is a $(\tau_1,\tau_2,\delta_1,\delta_2)$-ideal of a noetherian ring $R$, then $IS=SI$ is an ideal of $S$ and, letting $\tau_1,\tau_2,\delta_1$ and $\delta_2$ denote the induced functions on $R/I$, we have
\begin{equation}S/IS\cong (R/I)\left[ x,\tau_1,\delta_1\right]*_{R/I}(R/I)\left[y,\tau_2,\delta_2\right].\end{equation}
\end{theorem} 

\begin{proof}
Since $I$ is $(\tau_1,\tau_2,\delta_1,\delta_2)$-stable, we have $xI\subseteq Ix+I$, and $yI\subseteq Iy+I$.  Therefore $SI\subseteq IS+I\subseteq IS$. Next, we prove $SI\supseteq IS$. We have an ascending chain of ideals of $R$:
\begin{equation*}I\subseteq\tau_i^{-1}(I)\subseteq\tau_i^{-2}(I)\subseteq\tau_i^{-3}(I)\subseteq\cdots.\end{equation*}
Thus, for some $N\in\mathbb{N}$ sufficiently large, $\tau_i^{-N}(I)=\tau_i^{-N-1}(I)$. Applying $\tau_i^{N}$ to both sides yields $I=\tau_i^{-1}(I)$. Since $rx = x\tau_1^{-1}(r)-\delta_1\tau_1^{-1}(r)$  and $ry = x\tau_2^{-1}(r)-\delta_2\tau_2^{-1}(r)$ for all $r\in R$, we have $Ix\subseteq xI+I$, and $Iy\subseteq yI+I$. Hence, $IS\subseteq SI+I\subseteq SI$.
\end{proof}

\begin{corollary}\label{prime_ideal}If $I$ is a prime $(\tau_1,\tau_2,\delta_1,\delta_2)$-ideal of a noetherian ring $R$, then $IS=SI$ is a prime ideal of $S$ and $S/IS\cong (R/I)\left[ x,\tau_1,\delta_1\right]*_{R/I}(R/I)\left[y,\tau_2,\delta_2\right].$
\end{corollary}

Next we show that certain right denominator sets in the base ring $R$ extend to right denominator sets in the amalgamated product $S$. Sometimes an appropriate localization of $S$ will reduce to studying the case where the base ring $R$ is a division ring. 

First we recall that if $X$ is a right denominator set in a ring $R$, then $X$ is also a right denominator set in the Ore extension $R[x;\tau,\delta]$ provided $\tau$ is an automorphism and $\tau(X)=X$ (see e.g. \cite[Lemma 1.4]{Goodearl}). Furthermore, the identity map on the right ring of fractions $RX^{-1}$ extends to an isomorphism of $R[x;\tau,\delta]X^{-1}$ onto $(RX^{-1})[x;\overline{\tau},\overline{\delta}]$ sending $x1^{-1}$ to $x$, where $\overline{\tau}$ and $\overline{\delta}$ denote the induced maps on $RX^{-1}$ \cite[Lemma 1.4]{Goodearl}. We have an analogous result for amalgamated products. 

\begin{theorem}\label{5555} Let $X$ be a right denominator set in $R$ such that $\tau_1(X)=\tau_2(X)=X$. Then
\begin{enumerate}[(i)]
\item $X$ is a right denominator set in $S$, and
\item $SX^{-1}\cong (R X^{-1})[x;\tau_1,\delta_1]*_{R X^{-1}}(R X^{-1})[y;\tau_2,\delta_2],$
\end{enumerate}
where $\tau_1,\tau_2,\delta_1,\delta_2$ denote the induced maps on the right ring of fractions $R X^{-1}$.
\end{theorem}

\begin{proof}  Since $X$ is a right denominator set in $R$, then the right ring of fractions $RX^{-1}$ exists and we have
\begin{equation*}(R X^{-1})[x;\tau_1,\delta_1]*_{R X^{-1}}(R X^{-1})[y;\tau_2,\delta_2]=\sum_IRX^{-1}w_I.\end{equation*}
In the amalgamated product ring $(R X^{-1})[x;\tau_1,\delta_1]*_{R X^{-1}}(R X^{-1})[y;\tau_2,\delta_2]$ we have the following identities for every $m\in X$:
\begin{align*}
m^{-1}x &=x(\tau_1^{-1}(m))^{-1}+m^{-1}(\delta_1\tau_1^{-1}(m))(\tau_1^{-1}(m))^{-1},\\
m^{-1}y &=y(\tau_2^{-1}(m))^{-1}+m^{-1}(\delta_2\tau_2^{-1}(m))(\tau_2^{-1}(m))^{-1}.
\end{align*}
The two equations above show how $m^{-1}$ commutes past the variables $x$ and $y$. Thus, every element in $(R X^{-1})[x;\tau_1,\delta_1]*_{R X^{-1}}(R X^{-1})[y;\tau_2,\delta_2]$ can be written as a sum of the form $\sum_ir_iw_im_i^{-1}$.
Since $X$ is a right Ore set in $R$, then for every $r_1,r_2\in R$, $m_1,m_2\in X$, and $w_1,w_2$ words in the letters $x,y$, there exist $r_3\in R$ and $m_3\in X$ so that
\begin{align*}
r_1w_1m_1^{-1}+r_2w_2m_2^{-1}&=(r_1w_1m_1^{-1}m_2+r_2w_2)m_2^{-1}\\
&=(r_1w_1r_3m_3^{-1}+r_2w_2)m_2^{-1}\\
&=(r_1w_1r_3+r_2w_2m_3)m_3^{-1}m_2^{-1}\\
&=(r_1w_1r_3+r_2w_2m_3)(m_2m_3)^{-1}\in S(m_2m_3)^{-1}.
\end{align*}
Therefore, right common denominators exist. Hence any $f\in (R X^{-1})[x;\tau_1,\delta_1]*_{R X^{-1}}(R X^{-1})[y;\tau_2,\delta_2]$ can be written in the form $sm^{-1}$ for some $s\in S$ and $m\in X$.  
\end{proof}

We end this section by remarking that if $R$ is a noetherian domain, then the set $X:=R\backslash\{0\}$ satisfies the conditions in Theorem \ref{5555}. In this case, the right ring of fractions $SX^{-1}$ is isomorphic to an amalgamated product of Ore extensions over the division ring $RX^{-1}$.

\subsection{Amalgamated Products of Quadratic Extensions}\label{Quadratic Extensions}

We call an extension $R\subseteq A$ of rings \emph{quadratic} if there exists $x\in A\backslash R$ so that  $A$ is a free left $R$-module with basis $\{1,x\}$. Since $xR\subseteq Rx+R$, it follows that there exists a ring endomorphism $\tau:R\to R$ and a left $\tau$-derivation $\delta:R\to R$ so that $xr=\tau(r)x+\delta(x)$ for every $r\in R$. Furthermore $x^2=ax+b$ for some $a,b\in R$. Thus $A$ is isomorphic to the factor ring $R[x;\tau,\delta]/\left<x^2-ax-b\right>$. However there are certain compatibility conditions involving the elements $a,b\in R$ and the endomorphism $\tau:R\to R$ that must hold. To make this precise we observe that for every $p\in R[x;\tau,\delta]$ there exist unique $r_0,r_1\in R$ and $f\in R[x;\tau,\delta]$ so that $p=r_0+r_1x+f(x^2-ax-b)$. This implies that for every $r\in R$, the following identities hold in $R[x;\tau,\delta]$:
\begin{align}
&(x^2-ax-b)r= \tau^2(r)(x^2-ax-b),\label{s1}\\
&(x^2-ax-b)x=(x+\tau(a)-a)(x^2-ax-b).\label{s2}\end{align}
If $\tau$ is an automorphism, Eqns. \ref{s1}-\ref{s2} are equivalent to $x^2-ax-b$ being a normal element in $R[x;\tau,\delta]$.

Throughout this section, let $Q_1$ and $Q_2$ be arbitrary quadratic extensions of $R$. We will write them in the form $R[x;\tau_1,\delta_1]/I_1$ and $R[y;\tau_2,\delta_2]/I_2$ respectively, where $I_1=\left<x^2-ax-b\right>$, $I_2=\left<y^2-cy-d\right>$, and the triples $(a,b,\tau_1)$ and $(c,d,\tau_2)$ satisfy compatibility conditions analogous to those described in Eqns.\ref{s1}-\ref{s2} above. For an element $f+I_i \in Q_i$ ($i=1,2$), we write it simply as $f$. Let $Q=Q_1*_RQ_2$.
\\

In this section, $\tau_1$ and $\tau_2$ do not necessarily need to be automorphisms of $R$. Unless stated otherwise, they are only assumed to be endomorphisms.
\\

Let $\ell\geq 0$. We define the following alternating-letter words of length $\ell$ in $Q$:
\begin{align*}
&x^{(\ell)}=xyxyx\cdots, & &y^{(\ell)}=yxyxy\cdots, & &\widehat{x}^{(\ell)}=\cdots xyxyx, & &\widehat{y}^{(\ell)}=\cdots yxyxy.
\end{align*}

\begin{theorem}
The ring $Q$ is a free left $R$-module with basis $\{1,x^{(1)},y^{(1)},x^{(2)},y^{(2)},...\}$ (or equivalently $\{1,\widehat{x}^{(1)},\widehat{y}^{(1)},\widehat{x}^{(2)},\widehat{y}^{(2)},...\}$).
\end{theorem}

\begin{proof}
The relations $x^2=ax+b$, $y^2=cy+d$, $xr=\tau_1(r)x+\delta_1(r)$, $yr=\tau_2(r)y+\delta_2(r)$ ($r\in R$) imply that $\{1,x^{(1)},y^{(1)},x^{(2)},y^{(2)},...\}$ is a spanning set for $Q$. We need to prove $R$-linear independence. Let $A$ be the free left $R$-module having basis $\{1,f_1,f_2,...,g_1,g_2,...\}$ and set $E:=\text{End}_{\mathbb{Z}}(A)$. We view $E$ as a left $Q$-module by first having each $r\in R$ act via left multiplication. For 
$\displaystyle{p=c_0+\sum_{i>0}\left(c_if_i+d_ig_i\right)\in A}$, we define the actions of $x$ and $y$ by
\begin{align*}
x.p&:=\delta_1(c_0)+\tau_1(c_1)b+\Big(\tau_1(c_0)+\tau_1(c_1)a+\delta_1(c_1)\Big)f_1\\
&\phantom{===}+\sum_{i>0}\Big(\tau_1(d_i)+\delta_1(c_{i+1})+\tau_1(c_{i+1})a\Big)f_{i+1}+\Big(\delta_1(d_i)+\tau_1(c_{i+1})b\Big)g_i,\\
y.p&:=\tau_2(d_1)d+\delta_2(c_0)+\sum_{i>0}\Big(\delta_2(c_i)+\tau_2(d_{i+1})d\Big)f_i+\Big(\tau_2(c_{i-1})+\tau_2(d_{i})c+\delta_2(d_{i})\Big)g_{i}.
\end{align*}

One can verify that, as operators on $E$, we have $xr=\tau_1(r)x+\delta_1(r)$, $yr=\tau_2(r)y+\delta_2(r)$, $x^2=ax+b$, and $y^2=cy+d$. Therefore, these actions define a left $Q$-module structure on $E$. If $r_0+\sum_{i>0}r_ix^{(i)}+r_i^\prime y^{(i)}=0\in Q$ for some $r_i,r_i^\prime\in R$, then $(r_0+\sum_{i>0}r_ix^{(i)}+r_i^\prime y^{(i)}).1=r_0+\sum_{i>0} r_if_i+r_i^\prime g_i=0$. This implies $r_0,r_1,r_2,...,r_1^\prime,r_2^\prime,...$ are all $0$.  Thus $1,x^{(1)},y^{(1)},x^{(2)},y^{(2)},...$ are left $R$-linearly independent in $Q$.
\end{proof}

We remark that if $\tau_1$ and $\tau_2$ are automorphisms, then $Q$ is also a free right $R$-module with basis $\{1,x^{(1)},y^{(1)},x^{(2)},y^{(2)},...\}$. 

Next we prove an analogue of the Hilbert Basis theorem for $Q$. In the proof we make use of leading coefficients and leading terms.  However, in contrast to an ordinary polynomial ring, some elements of $Q$ can potentially have two leading coefficients instead of one. For instance if $p=\sum_{i=0}^{n}a_ix^{(i)}+b_iy^{(i)}\in Q$ ($a_i,b_i\in R$) with $a_n,b_n$ not both zero, then we say $p$ has degree $n$ (or $\text{deg}(p)=n$ for short). We call $a_nx^{(n)}+b_ny^{(n)}$ the \emph{leading term} of $p$. 

\begin{theorem}\label{Hilbert} If $R$ is right (left) noetherian and $\tau_1,\tau_2$ automorphisms of $R$, then $Q$ is right (left) noetherian.
\end{theorem}

\begin{proof}
First, let us suppose $R$ is right noetherian. Let $I$ be a right ideal of $Q$ and define
\begin{align*}
&L_1:=\{0\}\cup\{r\in R\mid\exists p\in I\text{ with leading term }rx^{(i)}\text{ for some }i\in\mathbb{N}\},\\
&L_2:=\{0\}\cup\{r\in R\mid\exists p\in I\text{ with leading term } a_ix^{(i)}+ry^{(i)} \text{ for some }a_i\in R,i\in\mathbb{N}\},
\end{align*}
First, we show $L_1$ is a right ideal of $R$. The proof that $L_2$ is a right ideal of $R$ is similar. Assume $\Lambda,\Lambda^\prime\in L_1$ are nonzero. Thus, there exist $p_\Lambda,p_{\Lambda^{\prime}}\in I$ having the form
\begin{equation*}p_{\Lambda^{\phantom{\prime}}} = \Lambda^{\phantom{\prime}}x^{(i)^{\phantom{\prime}}} +[\text{lower degree terms}],\end{equation*}
\begin{equation*}p_{\Lambda^{\prime}} = \Lambda^\prime x^{(i^\prime)} +[\text{lower degree terms}].\end{equation*}
Without any loss of generality, assume $i\leq i^\prime$. 

For every $m\geq\ell$, we let $x_{\ell,m},y_{\ell,m}\in Q$ be the unique alternating-letter words of length $m-\ell$ that satisfy the conditions $x^{(\ell)}x_{\ell,m}=x^{(m)}$ and $y^{(\ell)}y_{\ell,m}=y^{(m)}$. If $\Lambda+\Lambda^\prime =0$, then $\Lambda+\Lambda^\prime\in L_1$. On the other hand if $\Lambda+\Lambda^\prime\neq 0$, then $p_\Lambda x_{i,i^\prime}+p_{\Lambda^\prime}\in I$ has leading term $(\Lambda+\Lambda^\prime)x^{(i^\prime)}$ and it follows that $\Lambda+\Lambda^\prime\in L_1$. Furthermore, for any $r\in R$, we have $\Lambda r\in L_1$ because $$p_\Lambda(\tau^{{(i) }})^{-1}(r)  = \Lambda rx^{(i)}+[\text{lower degree terms}]\in I,$$
where $\displaystyle{\tau^{ (i) }=\underbrace{\tau_1\tau_2\tau_1\tau_2\tau_1\cdots}_{i \text{ terms}}}.$ (We define $\tau^{ (m)}$ similarly for every $m\in\mathbb{N}$.)

Therefore $L_1$ is a right ideal of $R$, hence finitely generated. Suppose $L_1=r_1R+\cdots +r_lR$ and $L_2=s_1R+\cdots +s_tR$. By multiplying on the right by appropriate words in $x$ and $y$, we find that for every $1\leq i\leq l$, $1\leq j\leq t$, there exist $f_i,g_j\in I$ having the form
\begin{align*}
&f_i = r_ix^{(N)} +\text{[lower degree terms]}\\
&g_j = a_jx^{(N)}+s_jy^{(N)} +\text{[lower degree terms]}
\end{align*}
for some $N\in\mathbb{N} $ sufficiently large and some $a_j\in R$. 

Next, define $M:=R+x^{(1)}R+y^{(1)}R+\cdots +x^{(N-1)}R+y^{(N-1)}R$. Since $M$ is a finitely generated right (and left) $R$-module, it follows that $M$ is noetherian. Thus, the submodule $(I\cap M)_R\subseteq M_R$ is finitely generated. Suppose $I\cap M=c_1R+\cdots +c_dR$.

Let $I_0$ be the right ideal of $Q$ generated by $f_1,...,f_l,g_1,...,g_t,c_1,...,c_d$. We show $I=I_0$. From the definition of $I_0$, it follows that $I_0\subseteq I$. If $p\in I$ and $\text{deg}(p)<N$, then $p\in I\cap M\subseteq I_0$. Thus, we suppose $p\in I$, $\text{deg}(p)=m\geq N$, and everything in $I$ having degree less than $m$ is in $I_0$. Let us assume
\begin{equation*}p=rx^{(m)} + sy^{(m)} +\text{[lower degree terms]}\in I\end{equation*}
for some $r,s\in R$. Thus, there exist $b_1,...,b_t\in R$ so that $s=s_1b_1+\cdots s_tb_t$.  Next, put 
\begin{equation*}u=\left(g_1\left(\tau_1^{-1}\tau^{ (N+1)}\right)^{-1}(b_1)+\cdots +g_t\left(\tau_1^{-1}\tau^{ (N+1)}\right)^{-1}(b_t)\right)y_{N,m}\in I_0.\end{equation*}
It follows that $u$ has the form $u=r^\prime x^{(m)}+sy^{(m)} +\text{[lower degree terms]}$ for some $r^\prime\in R$. Therefore 
\begin{equation*}p-u= (r-r^\prime )x^{(m)} +\text{[lower degree terms]}.\end{equation*}
Thus $r-r^\prime = r_1d_1+\cdots r_ld_l$ for some $d_1,...,d_l\in R$. Now define
\begin{equation*}u^\prime:=\left(f_1\left(\tau^{ (N)}\right)^{-1}(d_1)+\cdots f_l\left(\tau^{ (N)}\right)^{-1}(d_l)\right)x_{N,m}\in I_0.\end{equation*}
It follows that $u^\prime=(r-r^\prime )x^{(m)} +\text{[lower degree terms]}$. Therefore $p-u-u^\prime$ has degree less than $m$, hence $p-u-u^\prime\in I_0$. Therefore $p\in I_0$.

Now assume $R$ is left noetherian. To prove $Q$ is left noetherian, the argument is similar except now we construct left ideals $L_1$, $L_2$ of $R$ by writing polynomials in $Q$ with \emph{right-hand} coefficients and use the \emph{right} $R$-basis $\{1,\widehat{x}_1,\widehat{y}_1,\widehat{x}_2,\widehat{y}_2,....\}$ of $Q_R$.
\end{proof}

\begin{definition} Let $\tau$ be an endomorphism of a ring $R$. A left $\tau$-derivation $\delta$ is \emph{inner} if there exists $s\in R$ so that $\delta (r)=\tau(r)s-sr$ for all $r\in R$.
\end{definition}

\begin{theorem}\label{division} Suppose 
\begin{enumerate}[(i)]
\item $R$ is a division ring, 
\item $\tau_1$ and $\tau_2$ are automorphisms, and 
\item neither $\delta_i$ is an inner $\tau_i$-derivation.
\end{enumerate}
Then $Q$ is a principal ideal ring.
\end{theorem}

\begin{proof}
Let $n$ be the minimal degree among the nonzero elements of a nonzero proper ideal $I\subseteq Q$ and suppose $$f=x^{(n)}+a_{n-1}x^{(n-1)}+b_{n-1}y^{(n-1)}+\text{[lower degree terms]}\in I$$ for some $a_{n-1},b_{n-1}\in R$.
Then for all $r\in R$, we compute
\begin{align*}
fr-\tau^{ (n)}(r)f&=\Big(\tau^{ (n-1)}\delta_{[[n]]}(r)+a_{n-1}\tau^{ (n-1)}(r)-\tau^{ (n)}(r)a_{n-1}\Big)x^{(n-1)}\\
&\phantom{===}+\Big(\delta_1\tau_1^{-1}\tau^{(n)}(r)+b_{n-1}\tau_1^{-1}\tau^{ (n)}(r)-\tau^{ (n)}(r)b_{n-1}\Big)y^{ (n-1)}\\
&\phantom{===}+\text{[lower degree terms]}
\end{align*}
where $[[n]]=1$ (resp. $2$) when $n$ is odd (resp. even), and $\tau^{ (i)}=\underbrace{\tau_1\tau_2\tau_1\cdots}_{i\text{ terms}}$ for all $i\in\mathbb{N}$.

Since $\text{deg}(fr-\tau^{ (n)}(r)f)<n$ and $fr-\tau^{ (n)}(r)f\in I$, this implies the coefficients above are all zero. In particular, this will imply that $\delta_1$ is an inner $\tau_1$-derivation (because $\delta_1(r) = \tau_1(r)b_{n-1}-b_{n-1}r$ for all $r\in R$). Thus $f\notin I$.  Similarly, one can show that $I$ does not contain anything of the form $ y^{(n)}+\text{[lower degree terms]}.$

Let $I_n$ denote the set of elements of $I$ having degree $n$. From the previous arguments, it follows that every $g\in I_n$ has both leading coefficients nonzero. Hence, there exists $p=a_nx^{(n)}+b_ny^{(n)}+[\text{lower terms}]\in I$ with $a_n,b_n\in R$ both nonzero. If $h=c_nx^{(n)}+d_ny^{(n)}+[\text{lower terms}]\in I_n$ for some nonzero $c_n,d_n\in R$, then 
\begin{equation*}a_n^{-1}p-c_n^{-1}h=(a_n^{-1}b_n-c_n^{-1}d_n)y^{(n)}+[\text{lower terms}]\in I.\end{equation*}
Therefore $a_n^{-1}b_n-c_n^{-1}d_n=0$. Hence, we have
\begin{equation*}a_nc_n^{-1}h=a_nx^{(n)}+b_ny^{(n)}+[\text{lower terms}]\in I_n.\end{equation*}
It follows that $\text{deg}(p-a_nc_n^{-1}h)<n$ . Hence $p-a_nc_n^{-1}h=0$ and $I_n=(R\backslash\{0\})p$. 

We will show $I=\left<p\right>$.  Obviously, $\left<p\right>\subseteq I$.  Furthermore, if $p^\prime\in I$ and $\text{deg}(p^\prime)<n$, then $p^\prime=0$. Thus $p^\prime\in\left<p\right>$. If $\text{deg}(p^\prime)=n$, then $p^\prime\in Rp\subseteq \left<p\right>$. Now suppose $p^\prime = a_m^\prime x^{(m)}+b_m^\prime y^{(m)}+[\text{lower degree terms}]\in I$ with $m> n$ and assume everything in $I$ having degree less than $m$ is in $\left<p\right>$. Since 
\begin{equation*}\text{deg}(p^\prime-a_m^\prime a_n^{-1}px_{n,m}-b_m^\prime b_n^{-1}py_{n,m})<m,\end{equation*}
(recall that $x_{n,m},y_{n,m}\in Q$ are the unique alternating-letter words of length $m-n$ so that $x^{(n)}x_{n,m}=x^{(m)}$ and $y^{(n)}y_{n,m}=y^{(m)}$) this implies $p^\prime\in\left<p\right>$.
\end{proof}

\begin{theorem}\label{division 2} If $R$ is a division ring and $\tau_1$ and $\tau_2$ are automorphisms, then the one-sided ideals of $Q$ are either principal or doubly generated.
\end{theorem}

\begin{proof}
We will prove the left ideals of $Q$ are either principal or doubly generated. Here, we use the fact that $Q$ is a free left $R$-module having basis $\{1,\widehat{x}^{(1)},\widehat{y}^{(1)},\widehat{x}^{(2)},\widehat{y}^{(2)},...\}$ and write all polynomials with left-hand coefficients. The proof for the right ideals is similar; write all polynomials with right-hand coefficients. Suppose $I$ is a nonzero left ideal of $Q$. Let $n$ be the minimal degree among the nonzero elements of $I$. Choose $p\in I$ having degree $n$. We consider three cases. In all cases we show there exists $p^\prime\in I$ so that $I=Qp+Qp^\prime$. Whenever $I$ is a principal left ideal, $p^\prime$ may be chosen to be in $Qp$. In this situation, we do not mention $p^\prime$. We let $f$ be an arbitrary element of $I$ having degree $m$. If $m<n$, then $f=0$ and clearly $f\in Qp+Qp^\prime$. Thus, we let $m\geq n$ and assume everything in $I$ of degree less than $m$ is in $Qp+Qp^\prime$. Throughout this proof, we will make use of elements $\widehat{x}_{n,m},\widehat{y}_{n,m}\in Q$ (for $m\geq n$), which are defined by the conditions $\widehat{x}_{n,m}\widehat{x}^{(n)}=\widehat{x}^{(m)}$ and $\widehat{y}_{n,m}\widehat{y}^{(n)}=\widehat{y}^{(m)}$.

\emph{Case I: Every nonzero element of $I$ has no leading $\widehat{x}^{(i)}$-coefficients ($\widehat{y}^{(i)}$- coefficients).} Since $R$ is a division ring, we may without any loss of generality assume the leading term of $p$ is $\widehat{y}^{(n)}$ (or $\widehat{x}^{(n)}$). Suppose the leading term of $f$ is $r\widehat{y}^{(m)}$ (or $r\widehat{x}^{(m)}$) for some nonzero $r\in R$. Hence $\text{deg}\big[f-r(\widehat{x}_{n,m}+\widehat{y}_{n,m})p\big]<m$. Therefore, $I=Qp$.

\emph{Case II: Every element in $I$ of degree $n$ has no leading $\widehat{x}^{(i)}$-coefficient ($\widehat{y}^{(i)}$- coefficient) and there exists $p^\prime\in I$ having a nonzero leading $\widehat{x}^{(i)}$-coefficient ($\widehat{y}^{(i)}$-coefficient).} We will prove this for the $\widehat{x}^{(i)}$-case. In other words, we assume that every polynomial in $I$ of degree $n$ has $0$ as its leading $\widehat{x}^{(i)}$-coefficient and there exists $p^\prime\in I$ having a nonzero leading $\widehat{x}^{(i)}$-coefficient. The proof for the $\widehat{y}^{(i)}$-case is similar. Choose $p^\prime\in I$ of minimal degree satisfying the aforementioned condition. Without loss of generality, suppose 
\begin{align*}
&p=\widehat{y}^{(n)}+[\text{lower degree terms}],\\
&p^\prime= \widehat{x}^{(l)}+r\widehat{y}^{(l)}+[\text{lower degree terms}],\\
&f=s\widehat{x}^{(m)}+t\widehat{y}^{(m)}+[\text{lower degree terms}]
\end{align*} 
for some $r,s,t\in R$. It readily follows that $f-t\widehat{y}_{n,m}p=s\widehat{x}^{(m)}+[\text{lower degree terms}]$. If $s=0$, then $\text{deg}\Big(f-t\widehat{y}_{n,m}p\Big)<m$ and this implies $f\in Qp+Qp^\prime$. If $s\neq 0$, then $m\geq l$ and 
$\text{deg}\Big(f-t\widehat{y}_{n,m}p-s\widehat{x}_{l,m}p^\prime\Big)<m.$ Thus $f\in Qp+Qp^\prime$.

\emph{Case III: $p$ may be chosen with both leading coefficients nonzero.} Suppose without loss of generality
\begin{equation*}
p=\widehat{x}^{(n)}+r\widehat{y}^{(n)}+[\text{lower degree terms}]
\end{equation*}
for some nonzero $r\in R$.  Let $p^\prime\in I$ be a nonzero polynomial of minimal degree having $0$ as its leading $\widehat{x}^{(i)}$-coefficient (such a $p^\prime$ exists because $\widehat{y}_{n,n+1}r^{-1}p=\widehat{y}^{(n+1)}+[\text{lower degree terms}]\in I$, for instance). Since $R$ is a division ring, we may choose $p^\prime$ to be of the form $p^\prime = \widehat{y}^{(l)}+[\text{lower degree terms}]$. Suppose $f =s\widehat{x}^{(m)}+t\widehat{y}^{(m)}+[\text{lower degree terms}]$ for some $s,t\in R$. The leading term of  $f-s\widehat{x}_{n,m}p$ is $ty^{(m)}$. If $t=0$, then $\text{deg}\Big(f-s\widehat{x}_{n,m}p\Big)<m$. Therefore $f\in Qp+Qp^\prime$. If $t\neq 0$, then $m\geq l$ and $\text{deg}\Big(f-s\widehat{x}_{n,m}p-t\widehat{y}_{l,m}p^\prime\Big)<m$. Thus $f\in Qp+Qp^\prime$.
\end{proof}

\section{The Double Affine Hecke Algebra of Type $GL_2$: An Example of an Amalgamated Product}\label{DAHA}

In this section we show that the double affine Hecke algebra $\mathbb{H}_{q,t}(GL_2(k))$ associated to the general linear group $GL_2(k)$ is an amalgamated product of quadratic extensions over a three dimensional quantum torus provided $\text{char}(k)\neq 2$. From the results of Section \ref{Quadratic Extensions} it follows that $\mathbb{H}_{q,t}(GL_2(k))$ is noetherian.

The presentation we use for $\mathbb{H}_{q,t}(GL_2(k))$ is taken from \cite[Section 1.4.3]{Ch-book}. The algebra $\mathbb{H}_{q,t}(GL_2(k))$ can be defined over any base field $k$. However, in the results that follow we need $\text{char}(k)\neq 2$ because $2$ appears in a denominator.  Let $q,t\in k$ be fixed nonzero scalars such that $t^{1/2}$ exists. The double affine Hecke algebra $\mathbb{H}_{q,t}(GL_2(k))$ is defined as the associative $k$-algebra generated by invertible elements $T$, $X_1$, $X_2$, $Y_1$, $Y_2$ and has the defining relations
\begin{align}
& X_1X_2=X_2X_1, &&Y_1Y_2=Y_2Y_1,\\
&(T-t^{1/2})(T+t^{-1/2})=0, &&Y_2^{-1}X_1Y_2X_1^{-1}=T^2,\\
&T^{-1}Y_1T^{-1}=Y_2, &&TX_1T = X_2,\\
&Y_1Y_2X_1=qX_1Y_1Y_2, &&Y_1Y_2X_2=qX_2Y_1Y_2,\\
&X_1X_2Y_1=q^{-1}Y_1X_1X_2, &&X_1X_2Y_2=q^{-1}Y_2X_1X_2.
\end{align}

Let $R$ denote the $k$-algebra generated by the variables $z_1^{\pm 1},z_2^{\pm 1},z_3^{\pm 1}$ and having the defining relations $z_1z_2=z_2z_1$, $z_1z_3=q^{-1}z_3z_1$, and $z_2z_3=q^{-1}z_3z_2$. Let $\tau_1$ be the $k$-algebra automorphism of $R$ that interchanges $z_1$ and $z_2$, and sends $z_3$ to itself. Finally, let $\delta_1$ be the $k$-linear left $\tau_1$-derivation of $R$ defined by $-\alpha\frac{z_1+z_2}{z_1-z_2}(1-\tau_1)$, where $\alpha = \frac{1}{2}\left(t^{1/2}-t^{-1/2}\right)\in k$. Put $Q_1=R[x;\tau_1,\delta_1]/I_1$, where $I_1\subseteq R[x;\tau_1,\delta_1]$ is the ideal generated by the normal element $x^2-\left(\frac{t^{1/2}+t^{-1/2}}{2}\right)^2$. Therefore $Q_1$ is a quadratic extension of $R$. When it is not confusing, we will let $x$ denote the equivalency class $x+I_1\in Q_1$. Next, let $Q_2=R[y;\tau_2,\delta_2]/I_2$, where $\tau_2$ is the automorphism of $R$ given by $z_1\mapsto z_2$, $z_2\mapsto q^{-1}z_1$, $z_3\mapsto z_3$, $\delta_2\equiv 0$, and $I_2\subseteq R[y;\tau_2,\delta_2]$ is the ideal generated by the normal element $y^2-z_3^{-1}$. Thus $Q_2$ is a quadratic extension of $R$. Let $y$ denote the equivalency class $y+I_2\in Q_2$. 

Let $Q=Q_1*_RQ_2$. We have the following

\begin{theorem}\label{DAHA as an amalgamated product}
(for $\text{char}(k)\neq 2$) There is a $k$-algebra isomorphism $\varphi:Q\to\mathbb{H}_{q,t}(GL_2(k))$ which sends the generators of $Q$ to the following:
\begin{align*}
&z_1\mapsto X_1, & &z_2\mapsto X_2, & &z_3\mapsto Y_1Y_2, & &x\mapsto T-\alpha, & &y\mapsto Y_1^{-1}T,
\end{align*}
where $\alpha=\frac{1}{2}\left(t^{1/2}-t^{-1/2}\right)$.
\end{theorem}
\begin{proof}
It is straightforward to check that the map $\varphi$ above defines an algebra homomorphism. To show $\varphi$ is an isomorphism, we note that there is an algebra homomorphism $\widetilde{\varphi}:\mathbb{H}_{q,t}(GL_2(k))\to Q$ given by $X_1 \mapsto z_1$,  $X_2 \mapsto z_2$,  $Y_1 \mapsto  z_3\left(x+\alpha\right)y$,  $Y_2 \mapsto z_3y\left(x-\alpha\right)$, and $T \mapsto x+\alpha$. Finally, one can verify that $\varphi\widetilde{\varphi}=\text{Id}_{\mathbb{H}_{q,t}(GL_2(k))}$ and $\widetilde{\varphi}\varphi=\text{Id}_Q$.
\end{proof}

Therefore, we have the following

\begin{theorem} (for $\text{char}(k)\neq 2$) The double affine Hecke algebra $\mathbb{H}_{q,t}(GL_2(k))$ is a noetherian ring.\end{theorem}
\begin{proof} The quantum torus $R$ is a noetherian ring. By Theorems \ref{Hilbert} and \ref{DAHA as an amalgamated product}, $\mathbb{H}_{q,t}(GL_2(k))$ is noetherian also.
\end{proof}

\end{document}